\documentclass{amsart}
\usepackage{amsmath,amsfonts,amssymb,amsthm}
\usepackage{hyperref}

\theoremstyle{plain}
\newtheorem{theorem}{Theorem}[section]
\newtheorem{proposition}{Proposition}[section]

\newtheorem{remark}{Remark}[section]

\newtheorem{corollary}{Corollary}[section]

\theoremstyle{definition}
\newtheorem{definition}{Definition}[section]

\newtheorem*{remark*}{Remark}
\newtheorem*{conjecture*}{Conjecture}

\newcommand{\A}{\mathcal{A}}
\newcommand{\B}{\mathcal{B}}
\newcommand{\C}{\mathcal{C}}
\newcommand{\D}{\mathcal{D}}
\newcommand{\dom}{\text{dom}}

\newtheorem*{proposition*}{Proposition}{\bfseries}{\itshape}
\newtheorem*{problem*}{Problem}{\bfseries}{\rmfamily}

\begin{document}
%

\title{A Note on Computable Embeddings for Ordinals and Their Reverses}
\thanks{The first author was supported by Mathematical Center in Akademgorodok. The second author was supported by Bulgarian National Science Fund DN 02/16/19.12.2016 and NSF
  grant DMS 1600625/2016. The authors also would like to thank Manat Mustafa for his hospitality during their visit of Nazarbayev University.}

\author{N.\ Bazhenov and S.\ Vatev}

%

\address{Mathematical Center in Akademgorodok, Novosibirsk, Russia\ \and\ 
Sobolev Institute of Mathematics, Novosibirsk, Russia}
\email{bazhenov@math.nsc.ru}

\address{Department of Mathematical Logic and Applications\\
Sofia University\\
Bulgaria}
\email{stefanv@fmi.uni-sofia.bg}


%
\maketitle              
\begin{abstract}
  
  We continue the study of computable embeddings for pairs of structures, i.e. for classes containing precisely two non-isomorphic structures.
  Surprisingly, even for some pairs of simple linear orders, computable embeddings induce a non-trivial degree structure.
  Our main result shows that although $\{\omega \cdot 2, \omega^\star \cdot 2\}$ is computably embeddable in $\{\omega^2, {(\omega^2)}^\star\}$,
  the class $\{\omega \cdot k,\omega^\star \cdot k\}$ is \emph{not} computably embeddable in $\{\omega^2, {(\omega^2)}^\star\}$ for any natural number $k \geq 3$.
\end{abstract}

\section{Introduction}

The paper studies computability-theoretic complexity for classes of countable structures. A standard method of investigating this problem is to fix a particular notion of \emph{reduction} $\leq_r$ between classes, and then to gauge the complexity of classes via the degrees induced by $\leq_r$. 

One of the first examples of such reductions comes from descriptive set theory: Friedman and Stanley~\cite{FS89} introduced the notion of \emph{Borel embedding}. Informally speaking, a \emph{Borel embedding} $\Phi$ from a class $\mathcal{K}$ into a class $\mathcal{K}'$ is a Borel measurable function, which acts as follows. Given the atomic diagram of an arbitrary structure $A\in\mathcal{K}$ as an input, $\Phi$ outputs the atomic diagram of some structure $\Phi(A)$ belonging to $\mathcal{K}'$. The key property of $\Phi$ is that $\Phi$ is injective on isomorphism types, i.e. $\A \cong \B$ if and only if $\Phi(\A) \cong \Phi(\B)$.

Calvert, Cummins, Knight, and Miller \cite{CCKM04} (see also \cite{KMV07}) developed two different \emph{effective versions} of Borel embeddings. Roughly speaking, a \emph{Turing computable embedding} (or \emph{$tc$-embedding}, for short) is a Borel embedding, which is realized by a Turing functional $\Phi$. A \emph{computable embedding} is realized by an enumeration operator. It turned out that one of these notions is strictly stronger than the other: If there is a computable embedding from $\mathcal{K}$ into $\mathcal{K}'$, then there is also a $tc$-em\-bed\-ding from $\mathcal{K}$ into $\mathcal{K}'$. The converse is not true, see Section~\ref{sect:prelim} for formal details.

A powerful tool, which helps to work with Turing computable embeddings, is provided by the Pullback Theorem of Knight, Miller, and Vanden Boom~\cite{KMV07}. Informally, this theorem says that $tc$-embeddings behave well, when working with syntactic properties: one can ``pull back'' computable infinitary sentences from the output class $\mathcal{K}'$ to the input class $\mathcal{K}$, while \emph{preserving} the complexity of these sentences.

Nevertheless, Pullback Theorem and its consequences show that sometimes $tc$-em\-bed\-dings are too coarse: they \emph{cannot see}
finer structural distinctions between classes. One of the first examples of this phenomenon was provided by Chisholm, Knight, and
Miller~\cite{CKM07}: Let $VS$ be the class of infinite $\mathbb{Q}$-vector spaces, and let $ZS$ be the class of models of the
theory $\mathrm{Th}(\mathbb{Z},S)$, where $(\mathbb{Z},S)$ is the integers with successor. Then $VS$ and $ZS$ are equivalent with
respect to $tc$-em\-bed\-dings, but there is no computable embedding from $VS$ to $ZS$.

Another example of this intriguing phenomenon can be found in the simpler setting of classes generated by pairs of linear orderings, closed under isomorphism. Recall that by $\omega$ one usually denotes the standard ordering of natural numbers. For a linear order $L$, by $L^{\star}$ we denote the reverse ordering, i.e. $a\leq_{L^{\star}} b$ iff $b\leq_{L} a$. 

Ganchev, Kalimullin and Vatev~\cite{GKV} gave one such example.
For a structure $\A$, let $\tilde\A$ be the enrichment of $\A$ with a congruence relation $\sim$ such that every
congruence class in $\tilde\A$ is infinite and $\tilde\A /_\sim \cong \A$.
Then they showed that the class $\{\omega_S,\omega^\star_S\}$ is $tc$-equivalent to the class $\{\tilde\omega_S, \tilde\omega^\star_S\}$,
whereas $\{\tilde\omega_S, \tilde\omega^\star_S\}$ is not computably embeddable into $\{\omega_S,\omega^\star_S\}$.
Here $\omega_S$ and $\omega^\star_S$ are linear orderings of type $\omega$ and $\omega^\star$, respectively, together with the
successor relation.

One can prove (see, e.g., Theorem~3.1 in \cite{BazhGanVat-arxiv}) the following: Let $L$ be a computable infinite linear order with a least, but no greatest element. Then the pair $\{ L, L^{\star}\}$ is equivalent to $\{ \omega, \omega^{\star}\}$ with respect to $tc$-embeddings. 
This result gives further evidence that, in a sense, $tc$-embeddings cannot work with finer algebraic properties: Here a $tc$-embedding $\Phi$ can only employ the existence (or non-existence) of the least and the greatest elements. If one considers, say, the pair $\{ \omega^{\omega}, (\omega^{\omega})^{\star}\}$, then our $\Phi$ is not able to ``catch'' limit points, limits of limit points, etc. Section~\ref{subsect:bacground} gives a further discussion of interesting peculiarities of the pair $\{\omega, \omega^{\star}\}$.

On the other hand, when one deals with \emph{computable embeddings}, even finite sums of $\omega$ (together with their reverse orders) already exhibit a quite complicated structure: Let $k$ and $\ell$ be non-zero natural numbers. Then there is a computable embedding from $\{ \omega\cdot k, \omega^{\star}\cdot k\}$ into $\{ \omega\cdot \ell, \omega^{\star}\cdot \ell\}$ if and only if $k$ divides $\ell$ (Theorem~5.2 of~\cite{BazhGanVat-arxiv}). 
In other words, in this particular setting the only possible computable embeddings are the simplest ones~--- by appending a fixed
number of copies of an input order together.
We note that it is quite non-trivial to prove that all other embeddings $\Psi$ (e.g., a computable embedding from $\{ \omega \cdot 3,
\omega^{\star}\cdot 3\}$ to $\{ \omega \cdot 4, \omega^{\star}\cdot 4\}$) are not possible~--- our proofs fully employ the peculiarities inherent to
enumeration operators. These peculiarities have topological nature: indeed, one can establish the lack of continuous operators $\Psi$ (in the Scott topology).

The current paper continues the investigations of~\cite{BazhGanVat-arxiv}. We show that even adding the finite sums of $\omega^2$ (and their inverses) to the mix makes the resulting picture more combinatorially involved (compare with Theorem~5.2 mentioned above). 


The structure of the paper is as follows. Section~\ref{sect:prelim} contains the necessary preliminaries. In Section~\ref{sect:omega-times-2}, we give a simple computable embedding, which induces the following fact: for any non-zero $n$, there is a computable embedding from $\{\omega\cdot 2, \omega^{\star}\cdot 2\}$ into $\{ \omega^2\cdot n, (\omega^2)^{\star} \cdot n\}$ (Corollary~\ref{corol:omega-times-2}).

Sections~\ref{sect:case-3-2} and~\ref{sect:omega-times-k} together prove that for any $k\geq 3$, there is no computable embedding from $\{\omega \cdot k, \omega^{\star}\cdot k\}$ into $\{ \omega^2, (\omega^2)^{\star}\}$ (Theorem~\ref{theo:gen-case}). Note that Section~\ref{sect:case-3-2} gives a clear exposition for a simpler case $k=3$, and the other section deals with the general case. As a corollary, we obtain that $\{\omega \cdot 3, \omega^{\star}\cdot 3\}$ can be embedded into $\{ \omega^2\cdot n, (\omega^2)^{\star} \cdot n\}$ if and only if $n\geq 2$ (Corollary~\ref{corol:omega-times-3}).

Section~\ref{sect:powers} provides a first look at computable embeddings for powers of $\omega$ (and their inverses).

\section{Preliminaries}\label{sect:prelim}

We will slightly abuse the notations: both the set of natural numbers and the standard ordering of this set will be denoted by $\omega$. The precise meaning of the symbol $\omega$ will be clear from the context.
We consider only computable languages, and structures with domain contained in $\omega$. We assume that any considered class of structures
$\mathcal{K}$ is closed under isomorphism, modulo the restriction on domains. For a structure $\mathcal{S}$, $D(\mathcal{S})$ denotes the atomic
diagram of $\mathcal{S}$. We will often identify a structure and its atomic diagram.
We refer to atomic formulas and their negations as basic.


 Let $\mathcal{K}_0$ be a class of $L_0$-structures, and $\mathcal{K}_1$ be a class of $L_1$-structures. In the definition below, we use the following convention: An \emph{enumeration operator} $\Gamma$ is treated as a computably enumerable set of pairs $(\alpha,\varphi)$, where $\alpha$ is a finite set of basic $(L_0\cup \omega)$-sen\-ten\-ces, and $\varphi$ is a basic $(L_1\cup\omega)$-sentence. As usual, for a set $X$, we have $\Gamma(X) = \{ \varphi\,\colon (\alpha,\varphi) \in \Gamma,\ \alpha \subseteq X \}$.

\begin{definition}[{\cite{CCKM04,KMV07}}]
  An enumeration operator $\Gamma$ is a \emph{computable embedding} of $\mathcal{K}_0$ into $\mathcal{K}_1$, denoted by $\Gamma\colon \mathcal{K}_0 \leq_c \mathcal{K}_1$, if $\Gamma$ satisfies the following:
  \begin{enumerate}
  	\item For any $\A\in \mathcal{K}_0$, $\Gamma(\A)$ is the atomic diagram of a structure from $\mathcal{K}_1$.

  	\item For any $\A,\B\in \mathcal{K}_0$, we have $\A\cong\B$ if and only if $\Gamma(\A) \cong \Gamma(\B)$.
  \end{enumerate}	
\end{definition}

Any computable embedding has an important property of \emph{monotonicity}: If $\Gamma\colon \mathcal{K}_0 \leq_c \mathcal{K}_1$ and $\A\subseteq \B$ are structures from $\mathcal{K}_0$, then we have $\Gamma(\A) \subseteq \Gamma(\B)$ \cite[Proposition 1.1]{CCKM04}.

\begin{definition}[{\cite{CCKM04,KMV07}}]
  A Turing operator $\Phi=\varphi_e$ is a \emph{Turing computable embedding} of $\mathcal{K}_0$ into $\mathcal{K}_1$, denoted by $\Phi\colon \mathcal{K}_0 \leq_{tc} \mathcal{K}_1$, if $\Phi$ satisfies the following:
  \begin{enumerate}
  	\item For any $\A\in \mathcal{K}_0$, the function $\varphi^{D(\A)}_e$ is the characteristic function of the atomic diagram of a structure from $\mathcal{K}_1$. This structure is denoted by $\Phi(\A)$.

  	\item For any $\A,\B\in \mathcal{K}_0$, we have $\A\cong\B$ if and only if $\Phi(\A) \cong \Phi(\B)$.
  \end{enumerate}	
\end{definition}

\begin{proposition*}[Greenberg and, independently, Kalimullin; see \cite{Kal-18,KMV07}]
  If $\mathcal{K}_0 \leq_c \mathcal{K}_1$, then $\mathcal{K}_0 \leq_{tc} \mathcal{K}_1$. The converse is not true.
\end{proposition*}

Both relations $\leq_c$ and $\leq_{tc}$ are preorders. If $\mathcal{K}_0 \leq_{tc} \mathcal{K}_1$ and $\mathcal{K}_1\leq_{tc} \mathcal{K}_0$, then we say that $\mathcal{K}_0$ and $\mathcal{K}_1$ are \emph{$tc$-equivalent}, denoted by $\mathcal{K}_0 \equiv_{tc} \mathcal{K}_1$. For a class $\mathcal{K}$, by $\deg_{tc}(\mathcal{K})$ we denote the family of all classes which are $tc$-equivalent to $\mathcal{K}$. Similar notations can be introduced for the $c$-reducibility.
This paper is focused on the degree $\deg_{tc}(\{ \omega,\omega^{\star}\})$, the reader is referred to~\cite{BazhGanVat-arxiv} for a more detailed discussion of related results. 

We note that except the reductions $\leq_c$ and $\leq_{tc}$, there are many other approaches to comparing computability-theoretic complexity of classes of structures. These approaches include: transferring degree spectra and other algorithmic properties~\cite{HKSS02}, $\Sigma$-reducibility \cite{EPS-11,Puz-09}, computable functors \cite{HMMM-17,MPSS-18}, Borel functors~\cite{HTMM-18}, primitive recursive functors~\cite{BDKM,DHTK}, etc.


For two $\omega$-chains $\overline{x} = {(x_i)}^\infty_{i=0}$ and $\overline{y} = {(y_j)}^\infty_{j=0}$, analogous to the relation $\subseteq^\star$ between sets,
let us denote by $\overline{x} <^\star \overline{y}$ the following infinitary sentence
\[\bigvee_{q\in\omega}\bigwedge_{i,j>q}x_i < y_j.\]

The following proposition is essential for our results. It is a
slight reformulation of Proposition 5.7 from \cite{BazhGanVat-arxiv}.

\begin{proposition}\label{prop:no-chains}
  Suppose $\{\omega \cdot 2, \omega^\star \cdot 2\} \leq_c \{\C,\mathcal{D}\}$ via $\Gamma$, where
  $\C$ is a linear order without infinite descending chains and $\mathcal{D}$ is an infinite order without infinite
  ascending chains.
  Let $\A$ and $\B$ be copies of $\omega$ with mutually disjoint domains. 
  Then for any $\omega$-chains ${(x_i)}^\infty_{i=0}$ and ${(y_i)}^\infty_{i=0}$ such that $\Gamma(\A) \models \bigwedge_{i\in\omega} x_i<x_{i+1}$ and
  $\Gamma(\B) \models \bigwedge_{i\in\omega} y_i < y_{i+1}$, it follows that 
  \[\Gamma(\A + \B) \models \overline{x} <^\star \overline{y}\ \lor\ \overline{y} <^\star \overline{x}.\]
\end{proposition}

\subsection{Further background} \label{subsect:bacground}

This paper is focused on the degree $\deg_{tc}(\{ \omega,\omega^{\star}\})$. Historically speaking, the choice of this particular degree was motivated by the following open question:

\begin{problem*}[Kalimullin]
	It is easy to show that the pairs $\{ \omega, \omega^{\star}\}$ and $\{ \tilde\omega, \tilde\omega^{\star}\}$ are $tc$-equivalent. Moreover, $\{ \omega, \omega^{\star}\} \leq_c \{ \tilde\omega, \tilde\omega^{\star}\}$. Is there a computable embedding from $\{ \tilde\omega, \tilde\omega^{\star}\}$ to $\{ \omega, \omega^{\star}\}$?
\end{problem*}

This problem was a starting point of investigations of~\cite{BazhGanVat-arxiv} and the current paper. One can attack the problem via employing model-theoretic properties of the structures (in a way similar to~\cite{CKM07}). In particular, a naive way to distinguish these pairs would be the following. Each of the orders $\omega$ and $\omega^{\star}$ is rigid, while both $\tilde\omega$ and $\tilde\omega^{\star}$ have continuum many automorphisms. Maybe, this fact can help us to prove that $\{ \tilde\omega, \tilde\omega^{\star}\} \nleq_c \{ \omega, \omega^{\star}\}$? Nevertheless, this is \emph{not} the case~--- one can show that $\{ \tilde\omega, \tilde\omega^{\star}\} \equiv_c \{ (\omega^2, B), (\omega\cdot\omega^{\star},B) \}$, where $B$ is the standard block relation on a linear order. Since the structures $(\omega^2, B)$ and $(\omega\cdot\omega^{\star},B)$ are both rigid, it seems that studying automorphism groups does not help in this setting.

We note that quite unexpectedly (at least for us), the theory of Turing computable embeddings found applications in algorithmic learning theory. Section~3.2 of~\cite{BFS-arxiv} establishes connections between $tc$-embeddings and a particular paradigm of learnability for classes of countable structures. Informally speaking, this paradigm employs a learner whose goal is, given the atomic diagram of a structure $\mathcal{A}$, to learn the isomorphism type of $\mathcal{A}$. The learner is allowed to use both positive and negative data provided by the atomic diagram. Remarkably, the family $\{\omega, \omega^{\star}\}$ is learnable by a computable learner. We conjecture that our results can be also connected to learnability, specifically to its topological aspects (see, e.g., \cite{dBY}). The reader is referred to~\cite{BazhGanVat-arxiv} for more results on $\deg_{tc}(\{ \omega,\omega^{\star}\})$. 

\section{Positive results} \label{sect:omega-times-2}

Let $\A$ be a linear ordering and let us have, for all $a \in \A$, the linear orderings $\B_a$ with mutually disjoint domains.
Following Rosenstein~\cite{Ros82}, we define the generalized sum $\C = \sum_{a\in\A} \B_a$ as the linear ordering such that $\dom(\C) =
\bigcup_{a\in\A} \dom(\B_a)$ and for any $x,y \in \C$, we define $x <_\C y$ iff $x,y \in \B_a$ for some $a \in \A$ and $x <_{\B_a} y$, or
$x \in \B_a$, $y \in \B_{a'}$ and $a <_\A a'$.

\begin{theorem}\label{th:powers-simple}
  For any natural number $n \geq 1$,
  \[\{\omega \cdot n, \omega^\star \cdot n\} \leq_c \{\omega^2 \cdot n, {(\omega^2)}^\star \cdot n\}.\]
\end{theorem}
\begin{proof}
  The same enumeration operator $\Gamma$ works for all $n \geq 1$.
  For a linear ordering $\mathcal{L}$ and $a \in \mathcal{L}$, let $\mathcal{L}_a$
  be the linear ordering consisting of pairs $(a,b)$, where $b \in \mathcal{L}$, and
  ordered by the second component as in $\mathcal{L}$.
  Informally, for each element $a$ in the input linear order $\mathcal{L}$,
  the enumeration operator outputs $\mathcal{L}_a$.
  Moreover, all pairs in $\Gamma(\mathcal{L})$ are ordered lexicographically by the order induced by $\mathcal{L}$.
  In other words, $\Gamma(\mathcal{L}) = \sum_{a\in \mathcal{L}} \mathcal{L}_a$.
  \begin{itemize}
  \item
    If $\mathcal{L} \cong \omega \cdot n$, then
    \[\Gamma(\mathcal{L}) \cong \sum_{j\in n}\sum_{i\in\omega} \omega \cdot n = \sum_{j\in n}\omega^2 = \omega^2 \cdot n.\]
  \item
    If $\mathcal{L} \cong \omega^\star \cdot n$, then
    \[\Gamma(\mathcal{L}) \cong \sum_{j\in n}\sum_{i\in\omega^\star} \omega^\star \cdot n = \sum_{j\in n}{(\omega^2)}^\star = {(\omega^2)}^\star \cdot n.\]
  \end{itemize}
\end{proof}

For the next result, we need the following notation.
For a linear ordering $\mathcal{L}$ and an element $a$ in $\mathcal{L}$, we define
\begin{align*}
  \texttt{left}_{\mathcal{L}}(a) & = |\{b \in \dom(\mathcal{L}) \mid b \leq_{\mathcal{L}} a\}|\\
  \texttt{right}_{\mathcal{L}}(a) & = |\{b \in \dom(\mathcal{L}) \mid b \geq_{\mathcal{L}} a\}|\\
  \texttt{rad}_{\mathcal{L}}(a) & = \min\{\texttt{left}_{\mathcal{L}}(a), \texttt{right}_{\mathcal{L}}(a)\}.
\end{align*}
Informally, we will show that there exists an enumeration operator $\Gamma$ which can ``guess'' whether an element $a$ in the input
linear ordering $\mathcal{L}$ has finite or infinite radius, denoted $\texttt{rad}_{\mathcal{L}}(a)$.

\begin{theorem}\label{th:omega-2-to-square}
  $\{\omega \cdot 2, \omega^\star \cdot 2\} \leq_c \{\omega^2, {(\omega^2)}^\star \}$.
\end{theorem}
\begin{proof}
  We informally describe the work of the enumeration operator $\Gamma$.
  Suppose we have as input the finite linear ordering $\mathcal{L} = a_0 < a_1 < a_2 < \cdots < a_n$.
  For each $a_i$ in $\mathcal{L}$, $\Gamma$
  outputs the pairs of the form $(a_i,a_j)$, where $a_j \leq_{\mathbb{N}} \texttt{rad}_{\mathcal{L}}(a_i)$,
  where $\leq_{\mathbb{N}}$ is the standard ordering of natural numbers.
  All pairs in the output structure are ordered in lexicographic order.
  This concludes the description of how $\Gamma$ operates.
  Now we have two cases to consider for the input structure $\A$.
  
  Suppose that $\A = \A_1 + \A_2$, where $\A_{1}$ and $\A_2$ are copies of $\omega$.
  If $a \in \A_1$ is its $k$-th least element, then $\texttt{rad}_{\A}(a) = k$ and hence $a$
  contributes at most $k$ pairs to $\Gamma(\A)$.
  If $a \in \A_2$, then clearly $\texttt{rad}_{\A}(a) = \aleph_0$ and hence
  $a$ contributes infinitely many pairs to $\Gamma(\A)$, forming a linear ordering of type $\omega \cdot 2$.

  We conclude that in this case
  \[\Gamma(\A) \cong \sum_{i \in \omega}i + \sum_{i\in\omega} \omega \cdot 2 \cong \omega + \omega^2 = \omega^2.\]
\item
  Suppose that $\A = \A_1 + \A_2$, where $\A_{1}$ and $\A_2$ are copies of $\omega^\star$.
  If $a \in \A_1$, then $a$ contributes infinitely many elements of type $\omega^\star \cdot 2$ in $\Gamma(\A)$.
  If $a \in \A_2$ is its $k$-th greatest element, then $a$ contributes at most $k$ pairs in $\Gamma(\A)$.
  We conclude that in this case
  \[\Gamma(\A) \cong \sum_{i\in\omega^\star} \omega^\star \cdot 2 + \sum_{i\in\omega^\star} i \cong {(\omega^2)}^\star + \omega^\star= {(\omega^2)}^\star.\]
\end{proof}

\begin{corollary}\label{cr:omega-n-omega-squared}
  For any natural number $n \geq 1$,
  \[\{\omega \cdot (n+1), \omega^\star \cdot (n+1)\} \leq_c \{\omega^2 \cdot n, {(\omega^2)}^\star \cdot n \}.\]
\end{corollary}
\begin{proof}
  We use the same enumeration operator $\Gamma$ as in Theorem~\ref{th:omega-2-to-square}.
  Suppose $\A~=~\A_0 + \A_1 + \cdots + \A_n$, where each $\A_i$ is a copy of $\omega$.
  Then if $a \in \A_0$ is the $k$-th least element,
  then $a$ contributes at most $k$ pairs in $\Gamma(\A)$.
  If $a \in \A_i$, where $i > 0$, then $a$ contributes infinitely many pairs of the type of $\A$ to $\Gamma(\A)$.
  It follows that
  \begin{align*}
    \Gamma(\A) & \cong \sum_{i\in\omega}i + \underbrace{\sum_{i\in\omega} \omega\cdot (n+1) + \cdots + \sum_{i\in\omega}\omega \cdot (n+1)}_{n}\\
               & = \omega + \omega^2 \cdot n \\
               & = \omega^2 \cdot n.
  \end{align*}
  Suppose $\A = \A_0 + \A_1 + \cdots + \A_n$, where each $\A_i$ is a copy of $\omega^\star$.
  Then if $a \in A_n$ is the $k$-th greatest element, then $a$ contributes at most $k$ pairs in $\Gamma(\A)$.
  If $a \in \A_i$, where $0 \leq i < n$, then $a$ contributes infinitely many pairs of the type of $\A$ to $\Gamma(\A)$.
  It follows that
  \begin{align*}
    \Gamma(\A) & \cong \underbrace{\sum_{i\in\omega^\star}\omega^\star\cdot (n+1) + \cdots + \sum_{i\in\omega^\star} \omega^\star\cdot (n+1)}_{n} + \sum_{i\in\omega^\star} i\\
               & = {(\omega^2)}^\star \cdot n + \omega^\star\\
               & = {(\omega^2)}^\star \cdot n.
  \end{align*}
\end{proof}

\begin{corollary} \label{corol:omega-times-2}
  For any natural number $n \geq 1$,
  \[\{\omega \cdot 2, \omega^\star \cdot 2\} \leq_c \{\omega^2 \cdot n, {(\omega^2)}^\star \cdot n \}.\]
\end{corollary}
\begin{proof}
  This is straightforward.
  Let $\Gamma$ be the enumeration operator from Theorem~\ref{th:omega-2-to-square}.
  Then for a natural number $n \geq 1$, the embedding will be obtained by the enumeration operator, which, for linear ordering $\A$,
  simply copies $n$ number of times the linear ordering $\Gamma(\A)$.
\end{proof}

\section{The case $\{\omega \cdot 3, \omega^\star \cdot 3\} \not\leq_c \{\omega^2, {(\omega^2)}^\star\}$}\label{sect:case-3-2}
In this section, towards a contradiction, assume $\Gamma : \{\omega \cdot 3, \omega^\star \cdot 3\} \leq_c \{\omega^2,
{(\omega^2)}^\star\}$.
Let $\B$ be a copy of $\omega \cdot 3$ (or the reverse ordinal). In general, for
a subordering $\A$ of $\B$, we may have that $\Gamma(\A)$ is not a linear ordering.
For example, we may have $x,y\in \Gamma(\A)$, but none of the sentences $x<y$ or $y < x$ are in $\Gamma(\A)$.
Suppose $\Gamma(\B) \models x < y$. Then we claim that there is no extension $\C$ of $\A$ such that $\Gamma(\C) \models y < x$.
In other words, although $\Gamma(\A)$ does not ``know'' the relation between $x$ and $y$, this relation is already fixed.
Assume there is such an extension $\C$ for which $\Gamma(\C) \models y < x$. By compactness of enumeration operators, we may suppose that $\C$ extends $\A$ by only finitely many elements.
We can find another finite extension $\D$ of $\A$ with $\dom(\D) \cap \dom(\B) = \dom(\A)$ and $\dom(\D) \cap \dom(\C) = \dom(\A)$ such that
$\Gamma(\D) \models x < y \lor y < x$. Now we use monotonicity.
If $\Gamma(\D) \models x < y$, then we must have $\Gamma(\C \cup \D) \models x<y\ \&\ y < x$.
If $\Gamma(\D) \models y < x$, then we must have $\Gamma(\B \cup \D) \models x<y\ \&\ y < x$. In both cases we reach a contradiction.


\begin{remark}
  It is safe to always suppose that if $\A$ is a linear ordering (or its corresponding reverse linear ordering),
  then $\Gamma(\A)$ is also a linear ordering.
\end{remark}

Let us denote by $a <^\infty b$ the computable infinitary sentence saying that there are infinitely many elements between $a$ and $b$.

\begin{proposition}\label{prop:omega-not-2}
  For any infinite and coinfinite set $A$, if
  there is a copy $\A$ of $\omega$ with $\dom(\A) = A$ such that $\Gamma(\A) \cong \omega^2$,
  then there is no copy $\B$ of $\omega$ with $\dom(\B) \subseteq \mathbb{N} \setminus A$ such that
  $\Gamma(\B) \cong \omega^2$.
\end{proposition}
\begin{proof}
  Assume that there are at least two copies $\A$ and $\B$ of $\omega$, with mutually disjoint domains, such that
  $\Gamma(\A) \cong \omega^2$ and $\Gamma(\B) \cong \omega^2$.
  Then we can fix the infinite sequences $\overline{a} = {(a_i)}^\infty_{i=0}$ and $\overline{b} = {(b_i)}^\infty_{i=0}$ such that
  \begin{align*}
    & \Gamma(\A) \models a_0 <^\infty a_1 <^\infty a_2 <^\infty \cdots \\
    & \Gamma(\B) \models b_0 <^\infty b_1 <^\infty b_2 <^\infty \cdots
  \end{align*}
  Then by Proposition~\ref{prop:no-chains}, we have $\Gamma(\A + \B) \models \overline{a} <^\star \overline{b}\ \lor\ \overline{b} <^\star \overline{a}$.
  It follows that $\Gamma(\A + \B)$ extends a copy of $\omega^2 \cdot 2$, which is a contradiction
  because by monotonicity of enumeration operators this would mean that there is a copy $\C$ of $\omega \cdot 3$
  extending $\A + \B$ such that $\Gamma(\C)$ extends $\omega^2 \cdot 2$.
\end{proof}

From now on, in this section, we suppose that we work with copies $\A$ of $\omega$
such that $\Gamma(\A)$ has type \emph{strictly less} than $\omega^2$, i.e.
there exist natural numbers $n$ and $\ell$ such that $\Gamma(\A) \cong \omega \cdot n + \ell$.

\begin{proposition}\label{prop:upper-bound}
  There exists an infinite subset $D$ of natural numbers and a number $n$
  such that any copy $\A$ of $\omega$ with $\dom(\A) \subseteq D$ is such that $\Gamma(\A)$ has type at most $\omega \cdot n$.
\end{proposition}
\begin{proof}
  Towards a contradiction, assume that for any infinite subset $D$ of natural numbers, for any $n$, there exists a copy $\A$ of $\omega$ with $\dom(\A) \subseteq D$ such that $\Gamma(\A)$ is \emph{at least} $\omega \cdot n$.
  This means that we can consider a sequence $\A_n$ of copies of $\omega$, with mutually disjoint domains, such that
  $\Gamma(\A_n)$ has type \emph{at least} $\omega \cdot n$.
  Now we can partition each copy $\A_n$ into an infinite sum of finite parts ${(\alpha_{n,i})}^\infty_{i=0}$ such that $\A_n = \sum_{i\in\omega} \alpha_{n,i}$.
  Then we can form a new copy $\B$ of $\omega$ in the following way:
  \[\B = \sum_{i\in\omega} \sum^i_{n=0} \alpha_{n,i-n}.\]
  In other words, $\B = \alpha_{0,0} + \alpha_{0,1} + \alpha_{1,0} + \alpha_{0,2} + \alpha_{1,1} + \alpha_{2,0} + \alpha_{0,3} + \cdots$
  Then $\B$ contains $\A_n$ for all $n$ and by monotonicity,
  $\Gamma(\B)$ has type greater than $\omega\cdot n$ for all $n$.
  We conclude that $\Gamma(\B)$ has type at least $\omega^2$, which is a contradiction.
\end{proof}

\begin{remark}\label{remark:omega-fixed}
  Proposition~\ref{prop:upper-bound} allows us to proceed as in Section 7 of~\cite{BazhGanVat-arxiv} and suppose that we have fixed an infinite set $D$ and a number $n$
  such that any copy $\A$ of $\omega$ with $\dom(\A) \subseteq D$ is such that
  $\Gamma(\A) \cong \omega \cdot n$.
  From here on, all copies of $\omega$ that we consider will have as domains coinfinite subsets of $D$.
\end{remark}

\begin{proposition}\label{prop:last-chains-order}
  Let $\A$ and $\B$ be two such copies of $\omega$, with mutually disjoint domains, such that
  for the $\omega$-chains $\overline{a}_i = {(a_{i,j})}^\infty_{j=0}$ and $\overline{b}_i = {(b_{i,j})}^\infty_{j=0}$, where $i = 1,\dots,n$, we have
  \begin{align*}
    & \Gamma(\A) \models \overline{a}_1 < \overline{a}_2 < \cdots < \overline{a}_n\\
    & \Gamma(\B) \models \overline{b}_1 < \overline{b}_2 < \cdots < \overline{b}_n.
  \end{align*}
  Then
  \[\Gamma(\A+\B) \models \overline{a}_n <^\star \overline{b}_n.\]
\end{proposition}
\begin{proof}
  Assume not. By Proposition~\ref{prop:no-chains} we would have
  $\Gamma(\A+\B) \models \overline{b}_n <^\star \overline{a}_n$.
  Let $a_{n,0} \in \Gamma(\alpha)$ for some finite part $\alpha$ of $\A$.
  Then $\C = \alpha +\B$ is a copy of $\omega$ such that
  \[\Gamma(\C) \models \overline{b}_1 < \overline{b}_2 < \cdots < \overline{b}_n < a_{n,0}.\]
  It follows that $\Gamma(\C)$ extends a copy of $\omega \cdot n + 1$, 
  which is a contradiction with Remark~\ref{remark:omega-fixed}.
\end{proof}

\begin{proposition}\label{prop:no-merge-chains}
  Let $\A$, $\B$, and $\C$ be copies of $\omega$.
  Suppose that 
  \[\Gamma(\C) \models \overline{c}_1 < \overline{c}_2 < \cdots < \overline{c}_n,\]
  where $\overline{c}_i = {(c_{i,j})}^\infty_{j=0}$ are $\omega$-chains.
  Then there exists an infinite subsequence ${(i_s)}^\infty_{s=0}$  such that 
  \[\Gamma(\A + \B + \C) \models \bigwedge_{s \in \omega} c_{n,i_s} <^\infty c_{n,i_{s+1}}.\]
\end{proposition}
\begin{proof}
  Assume not. Then
  $\Gamma(\A+\B+\C) \models \overline{c}_1 < \cdots < \overline{c}_n + \mathcal{D}$,
  where $\D$ has the type of $\omega^2$.
  Let $d \in \D$ be such that $d \in \Gamma(\alpha + \beta + \C)$, where
  $\alpha$ and $\beta$ are finite parts of $\A$ and $\B$ respectively.
  Then $\alpha+\beta+\C$ is a copy of $\omega$, but 
  $\Gamma(\alpha+\beta+\C)$ extends a copy of $\omega \cdot n + 1$, which is a contradiction
  with Remark~\ref{remark:omega-fixed}.
\end{proof}

\begin{proposition}\label{prop:omega-3-omega-2}
  For any linear ordering $\mathcal{L}$ of type $\omega\cdot 3$,
  there is a linear ordering $\mathcal{M}$ of type $\omega \cdot 2$ with $\dom(\mathcal{L}) = \dom(\mathcal{M})$
  and $\Gamma(\mathcal{M}) \cong \omega^2$.
\end{proposition}
\begin{proof}
  Let $\mathcal{L} = \A + \B + \C$, where $\A$, $\B$, and $\C$ are copies of $\omega$.
  By Proposition~\ref{prop:no-merge-chains}, consider the infinite sequence $\overline{c} \in \Gamma(\C)$ such that
  \[\Gamma(\A + \B + \C) \models \bigwedge_{i\in\omega} c_i <^\infty c_{i+1}.\]
  Assume that for some finite parts $\alpha$ and $\beta$ of $\A$ and $\B$ respectively, for some $i$,
  $\Gamma(\alpha + \beta + \C) \models c_i <^\infty c_{i+1}$.
  But since $\alpha + \beta + \C$ is a copy of $\omega$, and $\Gamma(\C) \cong \omega \cdot n$,
  then $\Gamma(\alpha + \beta + \C)$ would extend a copy of $\omega \cdot (n+1)$, which is a contradiction with Remark~\ref{remark:omega-fixed}.
  It follows that any such finite parts $\alpha$ and $\beta$
  contribute finitely many elements to any interval of the form $(c_i,c_{i+1})$.
  
  Let $\overline{u}_i = {(u_{i,j})}^\infty_{j=0}$ be $\omega$-chains such that we can
  partition $\A$ and $\B$ into finite parts such that
  $\A = \sum_{i\in\omega} \alpha_i $ and $\B = \sum_{i\in\omega} \beta_i$ and for all $i$,
  \[\Gamma(\alpha_i+\beta_i+\C) \models \bigwedge^i_{j=0} c_j < u_{j,i-j} < c_{j+1}.\]
  Then, by monotonicity, we obtain the following:
  \[\Gamma(\sum_{i\in\omega}(\alpha_i+\beta_i) + \C) \models \bigwedge_{i\in\omega}\bigwedge_{j\in\omega}c_i < u_{i,j}< c_{i+1}.\]
  It follows that $\mathcal{M} = \sum_{i\in\omega}(\alpha_i+\beta_i) + \C$ is a copy of $\omega \cdot 2$ with $\dom(\mathcal{M}) =
  \dom(\mathcal{L})$ which produces a copy of $\omega^2$.
  
\end{proof}

\begin{proposition}\label{prop:strict-growth}
  Let $\mathcal{L}$ and $\mathcal{M}$ be disjoint copies of $\omega \cdot 2$ such that $\Gamma(\mathcal{L}) \cong \omega^2$ and $\Gamma(\mathcal{M}) \cong \omega^2$.
  Then there is a copy $\mathcal{N}$ of $\omega \cdot 3$ such that $\Gamma(\mathcal{N})$ extends a copy of $\omega^2 \cdot 2$.  
\end{proposition}
\begin{proof}
  Let $\mathcal{L} = \A + \B$ and $\mathcal{M} = \C + \D$, where $\A$, $\B$, $\C$ and $\D$ are copies of $\omega$.
  Let us fix the $\omega$-chains $\overline{b}_i = {(b_{i,j})}^\infty_{j=0}$ and $\overline{d}_i = {(d_{i,j})}^\infty_{j=0}$
  where $i = 1,\dots,n$ such that
  \begin{align*}
    & \Gamma(\B) \models \overline{b}_1 < \overline{b}_2 < \cdots < \overline{b}_n\\
    & \Gamma(\D) \models \overline{d}_1 < \overline{d}_2 < \cdots < \overline{d}_n.
  \end{align*}
  By Proposition~\ref{prop:no-merge-chains}, we can suppose that the $\omega$-chains $\overline{b}_n$ and $\overline{d}_n$ are such that
  \begin{align}
    & \Gamma(\A+\B) \models \bigwedge_{i\in\omega} b_{n,i} <^\infty b_{n,i+1} \label{eq:b}\\
    & \Gamma(\C+\D) \models \bigwedge_{i\in\omega} d_{n,i} <^\infty d_{n,i+1}, \label{eq:d}.
  \end{align}
  Now by Proposition~\ref{prop:last-chains-order} we have that 
  \begin{equation}
    \label{eq:5}
    \Gamma(\B + \D) \models \overline{b}_n <^\star \overline{d}_n.
  \end{equation}
  For an arbitrary partition of $\B$ and $\C$ into finite parts such that $\B = \sum_{i\in\omega}\beta_i$ and $\C =
  \sum_{i\in\omega} \gamma_i$, let us consider the copy $\mathcal{N}$ of $\omega \cdot 3$, where
  \[\mathcal{N} = \A + \sum_{i\in\omega}(\beta_i + \gamma_i)+ \D.\]
  By monotonicity, (\ref{eq:5}) implies that $\Gamma(\mathcal{N}) \models \overline{b}_n <^\star \overline{d}_n$.
  Now, again by monotonicity, (\ref{eq:b}) and (\ref{eq:d}) imply that $\Gamma(\mathcal{N})$ extends a copy of $\omega^2~\cdot~2$.
  
\end{proof}

Now we are ready to finish the proof.
Consider two disjoint copies $\mathcal{L}$ and $\mathcal{M}$ of $\omega\cdot 3$ such that $\Gamma(\mathcal{L}) \cong \omega^2$
and $\Gamma(\mathcal{M}) \cong \omega^2$.
By Proposition~\ref{prop:omega-3-omega-2}, we obtain two disjoint copies $\mathcal{L}_1$ and $\mathcal{M}_1$ of $\omega \cdot 2$
such that $\Gamma(\mathcal{L}_1) \cong \omega^2$ and $\Gamma(\mathcal{M}_1) \cong \omega^2$.
Then by Proposition~\ref{prop:strict-growth}, from $\mathcal{L}_1$ and $\mathcal{M}_1$ we can construct a copy $\mathcal{N}$ of $\omega
\cdot 3$ such that $\Gamma(\mathcal{N}) \not\cong \omega^2$.
Thus, we have proven the following theorem.
\begin{theorem}\label{th:omega-3-not-to-squared}
  $\{\omega \cdot 3, \omega^\star \cdot 3\} \not\leq_c \{\omega^2, {(\omega^2)}^\star\}$.
\end{theorem}

\begin{corollary} \label{corol:omega-times-3}
  For any non-zero natural number $n$,
  \[n \geq 2 \iff \{\omega \cdot 3, \omega^\star \cdot 3\} \leq_c \{\omega^2 \cdot n, {(\omega^2)}^\star \cdot n\}.\]
\end{corollary}
\begin{proof}
  First consider the direction $(\Rightarrow)$. For each $n \geq 2$ we will show how to build an enumeration operator $\Gamma_n$.
  Notice that by Corollary~\ref{cr:omega-n-omega-squared} we have an enumeration operator
  \[\Gamma_2: \{\omega\cdot 3, \omega^\star \cdot 3\} \leq_c \{\omega^2 \cdot 2, {(\omega^\star)}^2 \cdot 2\}.\]
  Moreover, by Theorem~\ref{th:powers-simple}, we have an enumeration operator  
  \[\Gamma_3 : \{\omega\cdot 3, \omega^\star \cdot 3\} \leq_c \{\omega^2 \cdot 3, {(\omega^\star)}^2 \cdot 3\}.\]
  
  Let $n = 2k$ for some $k \geq 1$.
  Then $\Gamma_{n}$ works so that, for any input $\A$, it outputs $k$ disjoint copies of $\Gamma_2(\A)$.

  Let $n = 2k + 3$ for some $k \geq 0$.
  Then $\Gamma_n$ works so that, for any input $\A$, it outputs $k$ disjoint copies of $\Gamma_2(\A)$
  together with a copy of $\Gamma_3(\A)$.
  The direction $(\Leftarrow)$ is exactly Theorem~\ref{th:omega-3-not-to-squared}.
  
\end{proof}

\section{The general case} \label{sect:omega-times-k}

Here, using the same techniques as in Section~\ref{sect:case-3-2}, we will obtain the following theorem.
\begin{theorem}\label{theo:gen-case}
  For any $k \geq 3$, $\{\omega \cdot k, \omega^\star \cdot k\} \not\leq_c \{\omega^2, {(\omega^2)}^\star\}$.
\end{theorem}

Again towards a contradiction, assume that we have fixed a number $k \geq 3$ and
an enumeration operator $\Gamma:\{\omega \cdot k, \omega^\star \cdot k\} \leq_c \{\omega^2, {(\omega^2)}^\star\}$.
Since Proposition~\ref{prop:omega-not-2} and Proposition~\ref{prop:upper-bound} still apply in this more general case,
we can use Remark~\ref{remark:omega-fixed} and suppose we have fixed a number $n$
such that we always work with copies $\A$ of $\omega$ such that $\Gamma(\A) \cong \omega \cdot n$.
By essentially repeating the proof of Proposition~\ref{prop:no-merge-chains},
we obtain the following proposition.
\begin{proposition}\label{prop:no-merge-chains}
  Let $\A_1,\A_2,\dots,\A_k$ be copies of $\omega$.
  Suppose that 
  \[\Gamma(\A_k) \models \overline{c}_1 < \overline{c}_2 < \cdots < \overline{c}_n,\]
  where $\overline{c}_i = {(c_{i,j})}^\infty_{j=0}$ are $\omega$-chains.
  Then there exists an infinite subsequence ${(i_s)}^\infty_{s=0}$  such that 
  \[\Gamma(\sum^k_{j=1} \A_j) \models \bigwedge_{s \in \omega} c_{n,i_s} <^\infty c_{n,i_{s+1}}.\]
\end{proposition}
The next proposition is a generalization of Proposition~\ref{prop:omega-3-omega-2}.
\begin{proposition}\label{prop:omega-k-omega-2}
  For any linear ordering $\mathcal{L}$ of type $\omega\cdot k$,
  there is a linear ordering $\mathcal{M}$ of type $\omega \cdot 2$ with $\dom(\mathcal{L}) = \dom(\mathcal{M})$
  and $\Gamma(\mathcal{M}) \cong \omega^2$.
\end{proposition}
\begin{proof}
  Let $\mathcal{L} = \sum^k_{i=1} \A_i$, where $\A_i$ are copies of $\omega$.
  By Proposition~\ref{prop:no-merge-chains}, consider the $\omega$-chain ${(c_{i})}^\infty_{i=0}$ in $\Gamma(\A_k)$ such that
  \[\Gamma(\sum^k_{j=1} \A_j) \models \bigwedge_{i\in\omega} c_i <^\infty c_{i+1}.\]

  As in the proof of Proposition~\ref{prop:omega-3-omega-2}, for any $\ell$,
  let $\overline{u}_\ell = {(u_{\ell,j})}^\infty_{j=0}$ be an $\omega$-chain such that
  we can partition $\A_i$ into finite parts with $\A_i = \sum_{j\in\omega} \alpha_{i,j} $, where $i = 1,2,\dots, k-1$, where
  for all $j$,
  \[\Gamma(\sum^{k-1}_{i=1}\alpha_{i,j} + \A_{k}) \models \bigwedge^j_{\ell=0} c_\ell < u_{\ell,j-\ell} < c_{\ell+1}.\]
  Then, by monotonicity, we obtain the following:
  \[\Gamma(\sum_{j\in\omega}\sum^{k-1}_{i=1}\alpha_{i,j} + \A_{k}) \models \bigwedge_{\ell\in\omega}\bigwedge_{j\in\omega}c_\ell < u_{\ell,j}< c_{\ell+1}.\]
  It follows that \[\mathcal{M} = \sum_{j\in\omega}\sum^{k-1}_{i=1}\alpha_{i,j} + \A_{k}\] is a copy of $\omega \cdot 2$ with $\dom(\mathcal{M}) =
  \dom(\mathcal{L})$ which produces a copy of $\omega^2$.
  
\end{proof}

Let us take two disjoint copies $\mathcal{L}$ and $\mathcal{M}$ of $\omega\cdot k$ such that $\Gamma(\mathcal{L}) \cong \omega^2$
and $\Gamma(\mathcal{M}) \cong \omega^2$.
By Proposition~\ref{prop:omega-k-omega-2}, we obtain two disjoint copies $\mathcal{L}_1$ and $\mathcal{M}_1$ of $\omega \cdot 2$
such that $\Gamma(\mathcal{L}_1) \cong \omega^2$ and $\Gamma(\mathcal{M}_1) \cong \omega^2$.
Then by Proposition~\ref{prop:strict-growth}, from $\mathcal{L}_1$ and $\mathcal{M}_1$ we can construct a copy $\mathcal{N}$ of $\omega
\cdot 3$ such that $\Gamma(\mathcal{N})$ extends a copy of $\omega^2 \cdot 2$.
By monotonicity, any copy $\mathcal{\hat{N}}$ of $\omega \cdot k$ extending $\mathcal{N}$ will be such that
$\Gamma(\mathcal{\hat{N}}) \not \cong \omega^2$.
We conclude that \[\{\omega \cdot k, \omega^\star \cdot k\} \not\leq_c \{\omega^2, {(\omega^2)}^\star\}.\]

\section{Positive Results for Powers of $\omega$} \label{sect:powers}

\begin{proposition}
  For any $n \geq 1$, $\{\omega^n,{(\omega^n)}^\star\} \leq_c \{\omega^{2n},{(\omega^{2n})}^{\star}\}$.
\end{proposition}
\begin{proof}
  Standard cartesian product construction as in \cite[Definition 1.40]{Ros82}.
\end{proof}

\begin{theorem}\label{th:power-2-to-3}
  $\{\omega^2, {(\omega^2)}^\star\} \leq_c \{\omega^3, {(\omega^3)}^\star\}$.
\end{theorem}
\begin{proof}
  The idea here is to replace each point by an interval of the form $[a,b]$, which means
  that this interval will have type $\omega \cdot k + \ell$ in the first case and
  $\ell + \omega^\star \cdot k$ in the second case.
  
  We informally describe the work of the enumeration operator $\Gamma$.
  Let us consider some finite diagram $\delta(\overline{a})$ of the input structure $\A$.
  For each $a$ in $\delta(\overline{a})$, $\Gamma$ executes the following steps:
  Find elements $b$ and $c$ such that $b \leq_\A a \leq_\A c$, where $b,c \leq_{\mathbb{N}} a$,
  such that $b$ is the $\leq_{\A}$-least such element and $c$ is the $\leq_{\A}$-greatest such element
  in $\delta(\overline{a})$.
  For all elements $d$ in $\delta(\overline{a})$ such that $b \leq_\A d \leq_\A c$, $\Gamma$
  enumerates in the output structure the pair $(a,d)$.
  All pairs are ordered lexicographically.
  This concludes the description of $\Gamma$.
  Now we have two cases to consider.
  
  Suppose that $\A = \sum_{i\in\omega}\A_i$, where $\A_i$ are copies of $\omega$.
  It is easy to see that for each $i$, there are only finitely many elements in $\A_i$,
  which contribute finitely many pairs in $\Gamma(\A)$.
  For instance, let $a$ be the $<_{\mathbb{N}}$-least element in $\A \setminus \A_0$.
  It follows that in $\A_0$ only the elements which are $<_{\mathbb{N}}$-less than $a$
  contribute finitely many pairs in $\Gamma(\A)$. We have
  \begin{align*}
    \Gamma(\A) & =  \sum_{a\in\A_0} (\omega\cdot k_{a,0} + \ell_{a,0}) + \cdots + \sum_{a \in \A_i} (\omega \cdot k_{a,i} + \ell_{a,i}) + \cdots\\
               & =  \omega^2 + \cdots + \omega^2 + \cdots\\
               & = \omega^3.
  \end{align*}

  For the second case, suppose that $\A = \sum_{i\in\omega^\star}\A_i$, where $\A_i$ are copies of $\omega^\star$.
  Again, for each $i$, there are only finitely many elements in $\A_i$,
  which contribute finitely many pairs in $\Gamma(\A)$.
  It follows that
  \begin{align*}
    \Gamma(\A) & =  \cdots + \sum_{a\in\A_i} (\ell_{a,i} + \omega^\star \cdot k_{a,i} ) + \cdots + \sum_{a \in \A_0} ( \ell_{a,0} + \omega^\star \cdot k_{a,0}) \\
               & =  \cdots + {(\omega^2)}^\star + \cdots + {(\omega^2)}^\star \\
               & = {(\omega^3)}^\star.
  \end{align*}
  
\end{proof}

\begin{corollary}
  For any $n \geq 1$, $\{\omega^n, {(\omega^n)}^\star\} \leq_c \{\omega^{2n-1}, {(\omega^{2n-1})}^\star\}$.
\end{corollary}
\begin{proof}
  We use the same enumeration operator $\Gamma$ as in Theorem~\ref{th:power-2-to-3}.
  Suppose $\A = \sum_{i\in\omega}\A_i$, where $\A_i$ are copies of $\omega^{n-1}$.
  As before, it is easy to see that in each $\A_i$, there are only finitely many
  elements whose contribution to $\Gamma(\A)$ form an ordinal of type $ < \omega^{n-1}$.
  All other infinitely many elements contribute to $\Gamma(\A)$ an ordinal of type $\omega^{n-1}\cdot k + \alpha$, for some $k < \omega$ and $\alpha < \omega^{n-1}$.
  Then
    
  \begin{align*}
    \Gamma(\A) & \cong \sum_{a\in\A_0} (\omega^{n-1} \cdot k_{a,0} + \alpha_{a,0}) + \cdots +\sum_{a\in\A_i}(\omega^{n-1} \cdot k_{a,i} + \alpha_{a,i}) + \cdots \\
               & = \omega^{n-1}\cdot \omega^{n-1} + \cdots + \omega^{n-1}\cdot\omega^{n-1} + \cdots\\
               & = \omega^{2n-2} \cdot \omega \\
               & = \omega^{2n-1}.
  \end{align*}
  The case when $\A \cong {(\omega^n)}^\star$ is similar.
\end{proof}

\begin{corollary}
  For any $n \geq 2$, $\{\omega^2,{(\omega^2)}^\star\} \leq_c \{\omega^n, {(\omega^n)}^\star\}$.
\end{corollary}
\begin{proof}
  For any natural number $n \geq 2$, we briefly describe the enumeration operator
  $\Gamma_n : \{\omega^2,{(\omega^2)}^\star\} \leq_c \{\omega^n, {(\omega^n)}^\star\}$.
  \begin{itemize}
  \item 
    If $n = 2k$, where $k \geq 1$, then for any input $\A$, $\Gamma_n$ outputs $\A^k$.
  \item
    If $n = 3$, then $\Gamma_3$ is the enumeration operator from Theorem~\ref{th:power-2-to-3}.
  \item
    If $n = 2k+3$, where $k \geq 1$, then for any input $\A$, $\Gamma_n$ outputs $\Gamma_3(\A) \cdot \A^k$.
  \end{itemize}
\end{proof}

\begin{proposition}\label{prop:power-omega}
  $\{\omega,\omega^\star\} \leq_c \{\omega^\omega, {(\omega^\omega)}^\star\}$.
\end{proposition}
\begin{proof}
  For input structure $\A$, replace each element $a$ in $\A$ by a copy of $\A^a$,
  where in the product we interpret $a$ as a natural number.
  If $\A \cong \omega$, then
  \[\Gamma(\A) \cong \sum_{i\in\omega} \omega^i = \omega^\omega.\]
  The case when $\A \cong \omega^\star$ is similar.
\end{proof}

\begin{corollary}
  For any natural number $n$,
  \[\{\omega^{n+1}, {(\omega^{n+1})}^\star\} \leq_c \{ \omega^{\omega+n}, {(\omega^{\omega+n})}^\star\}.\]
\end{corollary}
\begin{proof}
  We use the same enumeration operator as in Proposition~\ref{prop:power-omega}.
\end{proof}

\section{Future work}

We strongly conjecture that by employing the methods of this paper, one can prove that $\{ \omega\cdot 3, \omega^{\star}\cdot 3\} \nleq_c \{\omega^3, (\omega^3)^{\star}\}$. Furthermore, it would be interesting to consider pairs of structures $\{\mathcal{A}, \mathcal{B}\}$ such that $\mathcal{A}$ and $\mathcal{B}$ are \emph{not} linear orders, but still $\{\mathcal{A},\mathcal{B}\} \equiv_{tc} \{ \omega, \omega^{\star}\}$. We note that in this case, $\mathcal{A}$ and $\mathcal{B}$ \emph{cannot} be Boolean algebras (see Proposition~4.6 of~\cite{BFS-arxiv}).

\bibliographystyle{plain}
\bibliography{pairs}

\end{document}